\documentclass[12pt]{article}
\usepackage{geometry}
\usepackage{amsmath, amssymb, amsthm}
\usepackage{float}
\usepackage{natbib}
\usepackage{graphicx}
\usepackage[font=footnotesize]{caption}
\usepackage{subfigure} 

\usepackage{natbib}
\usepackage{booktabs}
\usepackage[utf8]{inputenc}

\usepackage{amsmath}   % needed for the "align" environment
\usepackage{amssymb}   % often helpful for math symbols
\usepackage{array}  
\usepackage{hyperref}

\usepackage{booktabs}  % optional, if you want nicer table lines
\usepackage{lipsum}    % just for dummy text, optional
\usepackage{array}       % Improves table formatting (for m{} column alignment)
\usepackage{multirow}    % Allows merging of table cells (for \multirow)
\usepackage{makecell}    % Enables line breaks within cells (for \makecell)
\usepackage{amsmath}     % Provides advanced math formatting (for aligned equations)
\usepackage{nicematrix}

\usepackage{algorithm}
\usepackage{algpseudocode}

\newtheorem{theorem}{Theorem}[section] % Define theorem 
\newtheorem{lemma}{Lemma}[section] % Define theorem environment
\newtheorem{definition}{Definition}[section]

\usepackage{media9}
\usepackage{multimedia}
\usepackage{amsthm}

\title{A graph based advection framework for climate-driven species distribution} 
  
\author{
Pranali Roy Chowdhury$^1$ \thanks{ Pranali Roy Chowdhury (\texttt{pranalir@iisc.ac.in})} \and
Soumyendu Raha$^1$ \thanks{Soumyendu Raha (\texttt{raha@iisc.ac.in})}
}

\date{}

\begin{document}

\maketitle
\noindent{}  $^1$Department of Computational \& Data Sciences, Indian Institute of Science Bengaluru, India

\bigskip

\begin{center}
{\bf Abstract}
\end{center}
Climate change is reshaping species interactions and movement across fragmented landscapes. Despite this, most mathematical models assume random diffusion, overlooking the influence of directed movement. Here, we develop a graph based reaction-diffusion-advection framework explicitly incorporating directional movement induced by environmental gradients. Our results show while diffusion promotes overall population persistence across the network, advective movement induces asymmetric flows. It create population hotspots by directing individuals toward optimal niches, often associated with nodes of high in-degree. We demonstrate the interplay between advection strength and network topology in determining species persistence. Strong advection increase local extinction risk by accumulating populations toward favorable nodes. Additionally, loss of ecological corridors can disrupt directed flow within the network, thereby restricting species from favorable patches. We found that this disruption might not cause immediate extinction, rather forcing species to spread to the suboptimal patches.  Our advection framework therefore efficiently captures how directional movement interacting with network topology governs species redistribution, hotspot formation, and predict extinction risk under environmental change.

\vspace{1.0cm}

\maketitle
\newpage

\section{Introduction}
% General introduction on climate change and examples of species movement due to climate change. Fragmented landscapes,..\\
% existing models- continuous, in networks how movement was studied. Why advection is missing. what is the motivation to consider advection.\\

% The role of climate-driven directional movement on ecological networks remains poorly understood.\\

%%%%%%%%%
Rising global temperatures, shifting precipitation patterns, and increasing frequency of extreme events are rapidly altering environmental conditions across the globe \cite{pinsky2019greater,deutsch2008impacts,urban2024climate}. These changes are not uniform over space, resulting in heterogeneous environmental conditions that vary across landscapes. These long-term alterations has impact on every possible living being on Earth. For instance, rising temperatures drive poleward and elevational range shifts \cite{chen2011rapid,pecl2017biodiversity,lenoir2020species}, altered precipitation regimes affect the survival, reproduction and influence their behaviour \cite{deguines2017precipitation}, and extreme events such as heatwaves and storms can cause abrupt population declines and local extinctions \cite{grant2017evolution,ummenhofer2017extreme}. The accelerated pace of climate change therefore is forcing the species to either adapt, or track the shifting climate or face extinction\cite{parmesan2003globally}. 

Dispersal of species is fundamental in determining biodiversity responses to rapid climate change \cite{travis2013dispersal,urban2016improving}. To persist, species often must track suitable climates into new regions through dispersal, colonization, and subsequent range shifts \cite{anderson2009dynamics}. Most modelling techniques developed assume that all species disperse randomly across any landscape. In particular, natural landscapes are increasingly fragmented due to both anthropogenic activities and environmental heterogeneity, restricting dispersal pathways and altering connectivity between habitat patches \cite{arancibia2022network}. Understanding how species redistribute across such fragmented systems under changing climatic conditions remains a fundamental challenge in ecology.  To bridge this gap, we develop a graph-based reaction-diffusion-advection model to investigate the interplay between network topology and climate-driven migration.

Temperature is a key environmental driver of species dispersal \cite{travis2013dispersal}. Despite extensive modelling of species movement in both continuous and discrete space setting \cite{cantrell2004spatial,cosner2014reaction,cantrell1991effects,neubert1995dispersal}, temperature-dependent movement remain largely underexplored with a few exceptions \cite{dornelas2024movement,maciel2013individual}.
 While diffusion-based ecological models provide a tractable and analytically rich framework, they implicitly assume unbaised movement and independent of environmental conditions \cite{mcpeek1992evolution,cantrell2006movement}. In contrast, empirical evidence suggests that many species exhibit directed movement in response to environmental gradients such as temperature, resource availability, or habitat quality, a behavior often described as taxis or advective movement \cite{chen2011rapid,parmesan2003globally}. For instance, laboratory and field studies have demonstrated thermotactic movement in ectotherms, where individuals actively track optimal temperature ranges \cite{angilletta2009thermal}, while large-scale observational studies reveal climate-driven range shifts in terrestrial and marine species consistent with directional dispersal toward favorable conditions\cite{chen2011rapid,pecl2017biodiversity,turchin1998quantitative} In continuous spatial models, this is typically incorporated as a drift term in advection–diffusion equations \cite{belgacem1995effects,cosner2003does,potapov2004climate,urban2012collision}. 

While most studies of directed movement are formulated in continuous spatial settings, real-world landscapes are inherently discrete and fragmented. Building on this need, network science provides a powerful framework to analyze ecological systems in heterogeneous and fragmented landscapes, where habitats are represented as nodes and links capture dispersal, interactions, or environmental connectivity \cite{dale2017applying}. It has been widely applied to study food webs, dispersal and migration, invasions, and disease spread, as well as species persistence under environmental gradients and climate warming \cite{fortin2021network,tylianakis2017ecological,bascompte2003nested,chen2023evolution,bdolach2025tipping,bhandary2023rising,hastings2026tipping}. Recent efforts further highlight that integrating diffusion-advection dynamics within network frameworks can improve predictions of species space use and persistence in changing landscapes \cite{prima2018combining}.

Despite its ecological relevance, the incorporation of advection into discrete or network-based models remains limited, particularly in fragmented landscapes where movement pathways are inherently structured. In this context, we develop a mechanistic framework in which advection emerges explicitly from temperature gradients, giving rise to directed edges in the underlying network. This formulation directly links dispersal to temperature-driven cues, rather than treating advection as a temperature dependent parameter. While temperature is widely recognized as a key driver of population dynamics, empirical evidence linking temperature directly to dispersal behavior remains limited \cite{amarasekare2024temperature}. To address this gap, we propose a framework in which movement is governed by thermal niche requirements and environmental sensing. As a baseline, we analyze a single-species model with temperature-dependent growth, providing a tractable setting to isolate the effects of climate-driven movement. 

By integrating a graph theoretical representation of landscapes with environmentally driven movement, our mechanistic formulation aim to capture directed dispersal of species along environmental gradients. Specifically, our work is motivated from the following research questions: How do environmental gradients shape the direction of movement and population distributions in fragmented landscapes; what is the role of network connectivity in driving the directional movement; How does the loss of ecological corridors alter the directed flow in the network and its affect on spatial distribution.  Our framework demonstrates how directed movement can be systematically constructed from environmental gradients using principles from discrete vector calculus. It enables us to quantify how the interplay between dispersal and warming influences population persistence, the emergence of population hotspots, and species redistribution under changing climatic conditions.

\section{Mathematical framework}
Due to anthropogenic activities, including agriculture and infrastructure development, large, unbroken tracts of forest areas have been broken into smaller, isolated fragments, which forms the patched ecosystems \cite{hanski1998metapopulation}. Under climate variability and environmental gradients, characterizing the heterogeneity among these patches becomes essential. Furthermore, understanding how this heterogeneity affects overall ecosystem resilience is a critical challenge. In such fragmented systems, dispersal of the species depends heavily on the connectivity of the patches, and species movement often exhibits a directional bias toward favorable conditions (e.g., cooler habitats or optimal resource availability). In this work, we propose a framework for modelling the directed movement of species based on the heterogeneity of the patches. 

We represent the patched ecosystem as an undirected graph $\mathcal{G} = (\mathcal{V}, \mathcal{E})$ where each vertices $\mathcal{V}=\{1,2,\cdots,N\}$ denotes the habitat patches in the ecosystem; and $\mathcal{E}$ denotes the unweighted edges representing the dispersal pathways between the habitat patches. Mathematically, the patched ecosystem therefore, can be fully described in the form of the adjacency matrix such that 
$\mathrm{A} = (a_{ij}) \in \mathbb{R}^{N \times N}$ defined by
\begin{equation}
        a_{ij} =
\begin{cases}
1, & \text{if distinct patches}\, i \text{ and}\, j\,\text{ are connected,} \\
0, & \text{otherwise},
\end{cases}
\end{equation}
and $a_{ii}=0.$ Clearly, the adjacency matrix $\mathrm{A}$ is symmetric and it encodes the probable dispersal pathways within the ecosystem. On the macroscopic scale, each patch show distinct environmental heterogeneity giving rise to node-feature which is rarely captured in the ecological models. Here, we incorporate local node-specifc ecological processes operating within each patch based on the node feature.

Let $u_i(t)$ denote the population density of the focal species in patch 
$i \in \mathcal{V}$ at time $t$. The intrinsic (node-level) dynamics in 
patch $i$ are described by a nonlinear reaction term 
$F_i(u_i, E_i),$ where $F_i$ represents the net growth rate of the species in patch $i$, which depend on local environmental conditions denoted by $E_i.$ The function $F_i$ captures within-patch biological interactions, 
including intrinsic growth, mortality, competition, predation, etc.

To incorporate spatial dispersal between habitat patches, we couple the 
node-level reaction dynamics through diffusion and advection processes 
defined on the graph $\mathcal{G}$. For clarity and completeness, 
we introduce the necessary terminology used in the formulation below.

\begin{definition}
Let $\mathcal{G} = (\mathcal{V}, \mathcal{E})$ be an undirected graph 
with adjacency matrix $\mathrm{A} = (a_{ij})$. 
For a node $i \in \mathcal{V}$, the set of \emph{neighbors} of $i$ 
is defined as
\[
\mathcal{N}(i) = \{\, j \in \mathcal{V} \; : \; a_{ij} = 1 \,\}.
\]
Therefore, $\mathcal{N}(i)$ consists of all nodes that are directly 
connected to node $i$ by an edge.
\end{definition}

\begin{definition}
Let $\mathcal{G} = (\mathcal{V}, \mathcal{E})$ be an undirected graph. 
The \emph{degree} of a node $i \in \mathcal{V}$, denoted by $d_i$, 
is the number of its neighbors. Mathematically,
\[
d_i = |\mathcal{N}(i)| 
= \sum_{j=1}^{N} a_{ij}.
\]
\end{definition}
For an undirected and unweighted graph, $d_i$ simply counts 
the number of habitat patches directly connected to patch $i.$

\begin{definition}
Let $\mathcal{G} = (\mathcal{V}, \mathcal{E})$ be a directed graph 
with adjacency matrix $\mathrm{A} = (a_{ij})$. 
For a node $i \in \mathcal{V}$, the in-degree and 
out-degree are defined respectively as
$$
d_i^{\mathrm{in}} = \sum_{j=1}^{N} a_{ij}, \qquad \text{and}\qquad
d_i^{\mathrm{out}} = \sum_{j=1}^{N} a_{ji}.
$$
\end{definition}

\begin{definition}
The principal eigenvalue of a matrix $M$, denoted by $\lambda_1$, is defined as the
\[
\lambda_1
=
\max \left\{ \operatorname{Re}(\lambda) : \lambda \in \sigma\big(M\big) \right\},
\]
where $\sigma\big(M\big)$ denotes the spectrum of the matrix $M$.
\end{definition}

\paragraph{Graph Laplacian:} The movement between habitat patches is modelled as unbiased dispersal along the edges. Mathematically, the graph Laplacian \cite{hamilton2020graph} provides a discrete analogue of the 
continuous diffusion operator. In particular, the Laplacian operator quantifies density-driven dispersal. That is, the diffusive flux at node $i$ is given by
\begin{equation}
    d \sum_{j \in \mathcal{N}(i)} (u_j - u_i),
\end{equation}
where $d>0$ denotes the diffusivity of the species. It represent the net movement of population from neighboring patches 
toward patch $i.$ Using the adjacency matrix $\mathrm{A}$ and the degree matrix $ \mathrm{D} = \mathrm{diag}(d_1, d_2, \dots, d_N),$ the graph Laplacian is defined as $-d\mathrm{L}$ where
\[
\mathrm{L} = \mathrm{D} - \mathrm{A}.
\]

\paragraph{Formulation of advection on graph:} To model directional movement driven by environmental gradients, 
we introduce an advection operator on a graph \cite{chapman2015advection,eliasof2024feature}. Though we take a general formulation, the environmental driver may represent temperature, 
resource availability, or habitat suitability associated with a particular node. 

Given a node $i$ in the graph $\mathcal{G},$ we denote the set of neighbors by $\mathcal{N}(i).$ The scalar field $\Theta \in \mathbb{R}^\mathrm{N}$ so that every node $i$ has a scalar value $\Theta_i$ associated with it. We then define the environmental gradient $\nabla \Theta_i$ in the node $i$ to be the directed edge $\nabla \Theta_i = (i,p(i))$ from $i$ to that neighbor $p(i)\in \mathcal{N}(i)$ which have minimum deviation from the optimum value of the species. That gives us
\begin{equation}
    p^u(i) = \mathrm{arg}\min_{j\in \mathcal{N}(i)}|\Theta_j-\Theta_{\mathrm{opt}}^u|
\end{equation}
where $\Theta^u_{\mathrm{opt}}$ denote the optimum environmental condition where the species can thrive. With this formulation, the directed edges were identified to form the directed subgraph $\mathcal{\Tilde{G}}$ of the original graph $\mathcal{G},$ where each node $i$ has at most one out-degree $d_i^\mathrm{out}.$  We give an example of an graph with 8 nodes where $\Theta_i$ is the temperature of the node $i$, such that $\Theta =[18,20,28,35,22,30,22,21]$ and $\Theta^u_{\mathrm{opt}}=20.$ The above discussed formulation gives the optimum directed edges $5\rightarrow2$ and $7\rightarrow 8$ of the original graph.
\begin{figure}[ht!]
    \centering
    \includegraphics[width=0.5\linewidth]{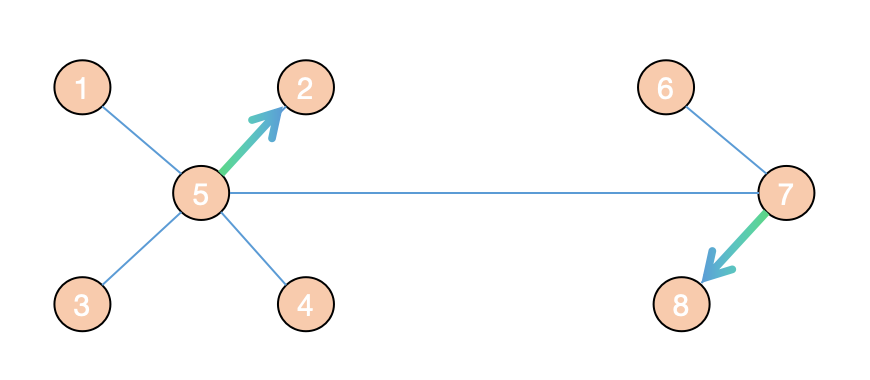}
    \caption{A schematic representation of the directed movement towards the optimum node value $\Theta^u_{\mathrm{opt}}.$}
    \label{fig:placeholder}
\end{figure}

We introduce the notion of discrete advection operator based on the principle of the continuous flux \cite{carmona2025discrete}. In the original continuous formulation the advection is given by
\begin{equation}
    \frac{du}{dt} =-\nabla\cdot(\mathbf{v}u)
\end{equation}
where $u$ is the scalar quantity (can be the population density of a species), $\nabla$ is the divergence operator, and $\textbf{v}$ is the velocity vector. Using the vector calculus identity
\begin{equation}
\mathbf{v} \cdot \nabla u
=
\nabla \cdot ( \mathbf{v} u)
-
u \, (\nabla \cdot \mathbf{v}),
\label{eq:identity}
\end{equation}
and integrating over the control volume we have
\begin{equation}\label{eq:divergence_thm}
\int_{\Omega} \mathbf{v} \cdot \nabla u \, d\Omega
=
\oint_{\partial \Omega} u \, \mathbf{v} \cdot d\mathbf{S}
-
\int_{\Omega} u \, (\nabla \cdot \mathbf{v}) \, d\Omega.
\end{equation}
We construct a discrete analogue of advection on the graph. We assume that each node $i$ has a unit control volume, and the orientation $v_{ij}$ represents the flux from $i \rightarrow j$ or outgoing from the node $i,$ and $v_{ji}$ is the flux incoming to $i.$ We approximate the off-diagonal entries of the advection matrix $\mathrm{A^{adv}=(\Tilde{a}_{ij})}$ by the boundary integral in \eqref{eq:divergence_thm}. This gives us the sum of the fluxes across the adjacent edges such that 
\begin{equation}\label{eq:off_diagonal_entries_adv}
\sum_{i\ne j}\mathrm{\Tilde{a}_{ij}} = \sum_{j\in \mathcal{N}(i),\, i\ne j} (v_{ij} - v_{ji} ).   
\end{equation}
The diagonal entries of the matrix $(\mathrm{\Tilde{a}_{ij}})$ is approximated by the volume integral involving the divergence. On a graph the divergence at node $i$ is approximated by the net inward
flux 
\begin{equation*}\label{eq:diagonal_entries_adv}
(\nabla\cdot\mathbf{v})_i
\approx
\sum_{j \in \mathcal{N}(i)} (v_{ij} - v_{ji}).
\end{equation*}
Combining the boundary and divergence contributions, the discrete advection operator satisfies
\begin{equation}\label{eq:_Disc_Advection_operator}
\sum_j \mathrm{\Tilde{a_{ij}}}u_j
=
\sum_{j \in \mathcal{N}(i)} (v_{ij} - v_{ji}) u_j
-
\left(
\sum_{j \in \mathcal{N}(i)} (v_{ji} - v_{ij})
\right)
u_i.
\end{equation}
\noindent We provide a pseudo-algorithm for the above formulation of the advection matrix. 

\begin{algorithm}
\caption{Construction of the directed flow and advection matrix}
\begin{algorithmic}

\State \textbf{Input:} $N$ (number of patches), 
$A$ (undirected adjacency matrix), 
$\Theta$ (environmental values), 
$\Theta_{\mathrm{opt}}^{u}$ (optimal environmental value)

\State \textbf{Output:} Directed movement matrix $A_u^{\mathrm{dir}}$ and advection matrix $A_u^{\mathrm{adv}}$

\For{$i = 1$ to $N$}
    \State $\mathcal{N}(i) \gets \{ j \mid A(i,j)=1 \}$ \Comment{Neighbors of patch $i$}
    
    \If{$\mathcal{N}(i)$ is empty}
        \State \textbf{continue}
    \EndIf

    \State Compute $d_u(j) = |\Theta_j - \Theta_{\mathrm{opt}}^{u}|$ for all $j \in \mathcal{N}(i)$
    
    \State $p_u(i) \gets \arg\min_{j \in \mathcal{N}(i)} d_u(j)$
    
    \State $A_u^{\mathrm{dir}}(i, p_u(i)) \gets 1$
\EndFor

\State $A_u^{\mathrm{adv}} \gets A_u^{\mathrm{dir}} - (A_u^{\mathrm{dir}})^{\top}$

\For{$i = 1$ to $N$}
    \State $A_u^{\mathrm{adv}}(i,i) \gets -\sum_{k=1}^{N} A_u^{\mathrm{adv}}(k,i)$
\EndFor

\end{algorithmic}
\end{algorithm}

\paragraph{Node-level reaction dynamics:} In a fragmented ecosystem of heterogeneous patch, the node-level reaction dynamics captures the temporal evolution of a particular species at each node of the network. It governs the intrinsic biological processes occurring within each node, such as reproduction, mortality, and intra- or interspecific interactions. The dynamics can be modelled as 
\begin{equation*}
    \dfrac{du_i}{dt} = F(u_i,\Theta_i)
\end{equation*}
where $u_i$ and $\Theta_i$ denote the species density and the environmental condition at node $i.$

\paragraph{Reaction-Diffusion-Advection Model on graph:}
Combining local dynamics, diffusion, and advection, the 
patched ecosystem model can be written in the compact form
\begin{equation}\label{eq:Compact_model}
    \frac{d \mathbf{u}}{dt}
=
\mathbf{F}(\mathbf{u}, \Theta)
- d \mathrm{L}\mathbf{u}
- \alpha \mathrm{A_u^{adv}(\Theta)}\,\mathbf{u}.
\end{equation}
where $\alpha$ and $d$ are the advection and diffusion rates, respectively.

\noindent We now show that the proposed advection operator has the following general properties.  

\begin{lemma}
    The advection operator in \eqref{eq:_Disc_Advection_operator} is mass conserving. 
\end{lemma}
\begin{proof}
    To show that our advection operator is mass-conserving we need to show that the total population does not change over time due to advection, that is $\dfrac{d}{dt}\sum_{i=1}^N u_i = 0.$ Let $u_i(t)$ represent the population density at the node $i$.\\
    Now \begin{equation}
        \begin{aligned}
            \dfrac{d}{dt}\sum_{i=1}^N u_i = \sum_{i=1}^N\dfrac{du_i}{dt} = \sum_{i=1}^N\left(\sum_{j=1}^N \mathrm{\Tilde{a_{ij}}}u_j\right)= \sum_{j=1}^N u_j \left( \sum_{i=1}^N \mathrm{\Tilde{a_{ij}}} \right)
        \end{aligned}
    \end{equation}
    By our construction of the matrix $\mathrm{A^{adv}}$ we have $\sum_{i=1}^N \mathrm{\Tilde{a_{ij}}}=0.$ Therefore, $\dfrac{d}{dt}\sum_{i=1}^N u_i = 0.$
\end{proof}
\begin{lemma}
The advection operator in \eqref{eq:_Disc_Advection_operator} is
generally asymmetric and non-normal.
\end{lemma}

\begin{proof}
Clearly, the matrix $\mathrm{A^{adv}}=(\mathrm{\Tilde{a}}_{ij})$ is asymmetric. A matrix $\mathrm{B}$ is normal if
\[
B B^{\top}
=
B^{\top}B.
\]
In our formulation the off-diagonal part of the advection matrix $\mathrm{\Tilde{a_{ij}}}$ is skew-symmetric,
but the diagonal entries are
\[
\widetilde a_{ii}
=
\sum_{k}(v_{ik}-v_{ki}),
\]
which are generally nonzero. Therefore, $\mathrm{A^{adv}} $ is not skew-symmetric. Since a real skew-symmetric matrix is normal,
the addition of a nonzero diagonal component destroys normality
unless the discrete divergence vanishes identically.
In the generic case 
\[
\sum_{k}(v_{ik}-v_{ki}) \neq 0.
\]
Thus, 
\[
\mathrm{A^{adv}}\, \mathrm{A^{adv}}^{\top}
\neq
\mathrm{A^{adv}}^{\top}\mathrm{A^{adv}},
\]
and hence the matrix $\mathrm{A^{adv}}$ is non-normal.
\end{proof}

%%%%%%%%%%%%%

%%%%%%%%%

\section{Results and ecological relevance}

Predicting how distribution of species is affected by the ongoing climate change, we first need to understand the impact of temperature on the growth and well-being of the species which translate into population dynamics. Temperature is one the most important factor which influences the physiological performance of species, especially ectotherms (organisms whose body temperature largely depends on the surrounding environment). This strong dependence makes them highly vulnerable to climate warming. Here, we model the dominant dependence of temperature on the intrinsic growth rate of the species. The per capita birth rate exhibits a unimodal response
to temperature given by the gaussian function \cite{scranton2017predicting,amarasekare2015effects}

\begin{equation}\label{eq:temperature_dependency}
    \gamma^u(\Theta) = 
\gamma^u_{\mathrm{opt}} e^{
-\frac{(\Theta_i^u - \Theta_{\mathrm{opt}}^u)^2}{2 s_u^2}},
\end{equation}
where $\gamma^u_{\mathrm{opt}}$ is the optimal temperature at which the growth rate is maximum and $s_u$ gives the temperature range over which the species show positive growth. Considering a temperature range with $\gamma^u_{\mathrm{opt}}=3,\,s_u=2$ and $ \Theta_{\mathrm{opt}}^u=25^\circ$ C, the the temperature response of per capita birth $\gamma^u(\Theta)$ is shown in the Fig.~\ref{fig:temperature_profile}. We will follow this temperature profile throughout the manuscript for our model \eqref{eq:Compact_model}. 

\begin{figure}[ht!]
    \centering
    \includegraphics[width=0.5\linewidth]{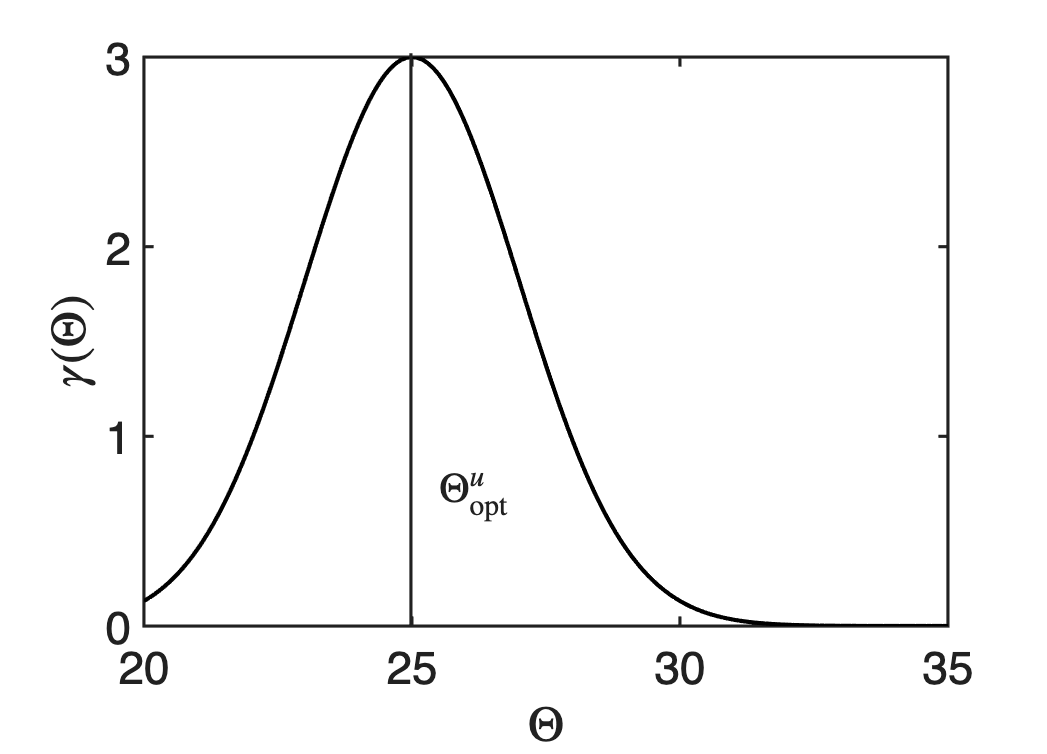}
    \caption{Temperature response of per capita growth rate.}
    \label{fig:temperature_profile}
\end{figure}

\subsection{Advection driven persistence criteria }
We consider the general ecological network system \eqref{eq:Compact_model}, where the local dynamics at node $i$ is given by $F(u_i,\Theta_i) = \gamma(\Theta_i)u(1-u)-\delta u.$ Here we consider the growth rate to be dependent on the node feature, in particular, temperature of the particular patch. The equilibrium densities of the network is the complete extinction state where $\textbf{u}_*=0$ and the existence state $\textbf{u}_* = 1-\dfrac{\delta}{\gamma(\Theta).}$ Note that the extinction equilibria always exists. We linearize the system around the equilibria $\textbf{u}_*=0$ and obtain the linearized system 
\begin{equation}\label{eq:Linearized_model}
    \frac{d \mathbf{u}}{dt}
=
\mathbf{M}(\Theta,d,\alpha) \mathbf{u},
\end{equation}
where 
\begin{equation}
    \mathbf{M}(\Theta,d,\alpha)  = \mathrm{R}-d \mathrm{L}
- \alpha \mathrm{A_u^{adv}(\Theta)}\,
\end{equation}
and the matrix $\mathrm{R} = \mathrm{diag}(\gamma(\Theta_i)-\delta),$ and $\mathrm{L},\, \mathrm{A_u^{adv}(\Theta)}$ are as constructed in previous section. 
Let
\[
M(\Theta, d, \alpha) = R - dL - \alpha A^{\mathrm{adv}},
\]
where $R = \mathrm{diag}(r_i - \delta)$, $L$ is the graph Laplacian, and $A^{\mathrm{adv}}$ is the advection matrix.

\begin{theorem}[Global Persistence]
     Let $\lambda_1(\Theta,\,d,\,\alpha)$ be the principal eigenvalue of the matrix $M(\Theta, d, \alpha),$ therefore, $\lambda_1(\Theta,\,d,\,\alpha)$ satisfies the following eigenvalue problem \begin{equation}
    M(\Theta, d, \alpha)\phi_i = \lambda_1(\Theta,\,d,\,\alpha)\phi_i.
\end{equation}
If $\lambda_1(\Theta,\,d,\,\alpha)< 0,$ then the extinction state $\mathbf{u}=0$ of the system \eqref{eq:Compact_model} is globally asymptotically stable, and if $\lambda_1(\Theta,\,d,\,\alpha)> 0$ extinction state is unstable and the population persists in all nodes at $\mathbf{u}_*.$ The critical persistence threshold is $\lambda_1(\Theta,\,d,\,\alpha)=0.$
\end{theorem}
In order to understand the node-specific persistence, we give the following theorem. 
\begin{theorem}[Node-specific Persistence]
Consider the advection only system
\[
\frac{d \mathbf{u}}{dt}
=
\mathbf{F}(\mathbf{u}, \Theta)
-
\alpha \mathrm{A_u^{adv}}(\Theta)\mathbf{u},
\qquad \alpha>0,
\]
where 
\[
F_i(u_i,\Theta_i)
=
\gamma(\Theta_i)u_i(1-u_i)
-
\delta u_i.
\]

Let us define $\gamma_i := \gamma(\Theta_i)$ and 
$r_i := \gamma_i - \delta.$ Then the $i^{th}$ node satisfies
\begin{equation*}
    \dot u_i
=
u_i(r_i - \alpha)
-
\gamma_i u_i^2
+
\alpha \sum_{j\to i} u_j.
\end{equation*}

If $i^{th}$ node receives positive equilibrium inflow such that 
\begin{equation}
    \alpha \sum_{j\to i} u_j^* > 0,
\end{equation}
then there exists a strictly positive equilibrium $u_i^*>0$
even if $r_i < 0.$ 
\end{theorem}

\begin{proof}
By construction for each node $i$,
\begin{equation}
    (\mathrm{A_u^{\mathrm{adv}}} \mathbf{u})_i
=
\sum_{j\to i} u_j
-
\sum_{i\to k} u_i.
\end{equation}
The directed matrix is constructed so that
each node has exactly one outgoing edge (whenever it has neighbors),
the outflow term satisfies
\[
\sum_{i\to k} u_i = u_i.
\]

At a positive equilibrium,
\[
0
=
u_i^*(r_i - \alpha)
-
\gamma_i (u_i^*)^2
+
\alpha \sum_{j\to i} u_j^*.
\]

We define the discriminant as 
$\Delta
=
(r_i - \alpha)^2
+
4\gamma_i \alpha \sum_{j\to i} u_j^*.$
Since $\gamma_i>0$ and $\alpha>0$, we have $\Delta>0$
whenever $\sum_{j\to i} u_j^*>0$. The positive root is given by
$u_i^*
=
\frac{
(r_i - \alpha)
+
\sqrt{(r_i - \alpha)^2
+
4\gamma_i \alpha \sum_{j\to i} u_j^*}
}
{2\gamma_i}.$
Since the term inside square root strictly exceeds $|r_i-\alpha|$
whenever $\sum_{j\to i} u_j^*>0$, the numerator is positive.
Hence $u_i^*>0$. Therefore, even when $r_i<0$, a positive equilibrium
exists for $\alpha>0$ provided inflow of population in node $i$ is nonzero.
\end{proof}

The Fig.~\ref{fig:network_random_distribution} depicts the numerical justification of the above theorem with diffusivity $d=1$ and advection strength $\alpha=1$. We present a random Watts-Strogatz network \cite{watts1998collective} network with N=100 nodes and 700 edges. The environmental variable (node feature) $\Theta$ is taken to be randomly distributed over the entire network. Our result shows that the distribution of the species is more towards the nodes having highest in-degrees. The population does not show extinction at any nodes because of the high diffusivity and strong advection considered in the model. Ecologically, the above theorem proves that even if the patch is locally unsuitable for species growth and reproduction, the
positive advective inflow from the neighboring nodes might generate a strictly positive equilibrium density. The high Pearson correlation coefficient ($r=0.969$) indicates that population accumulation predominantly occurs in highly connected nodes, suggesting that network topology strongly governs the spatial destribution of the species under advection. It follows directly from the inflow–outflow structure imposed by the construction of the advection matrix.

% \begin{figure}[ht!]
%     \centering
%     \subfigure[]{\includegraphics[width=0.3\linewidth]{Figures/one_species_random_E_distribution.png}}
%     \subfigure[]{\includegraphics[width=0.3\linewidth]{Figures/One_species_random_E_u_distrubution.png}}
%     \subfigure[]{\includegraphics[width=0.3\linewidth]{Figures/One_species_random_E_degree_u.png}}
%     \caption{(a) Random distribution of $\Theta$ across the network; (b) The accumulation of the density at the optimal nodes forming hotspots of population; (c) The dependence of the formation of hotspots based on the optimal thermal regime and ecological corridors. }
%     \label{fig:network_random_distribution}
% \end{figure}

\begin{figure}[ht!]
    \centering
    \subfigure[]{\includegraphics[width=0.35\linewidth]{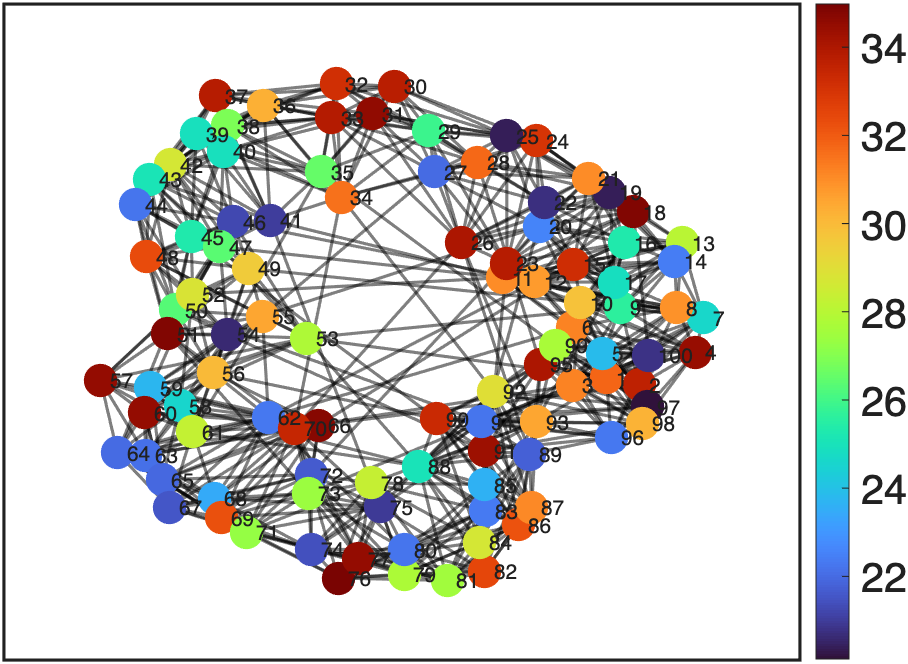}}
    \subfigure[]{\includegraphics[width=0.35\linewidth]{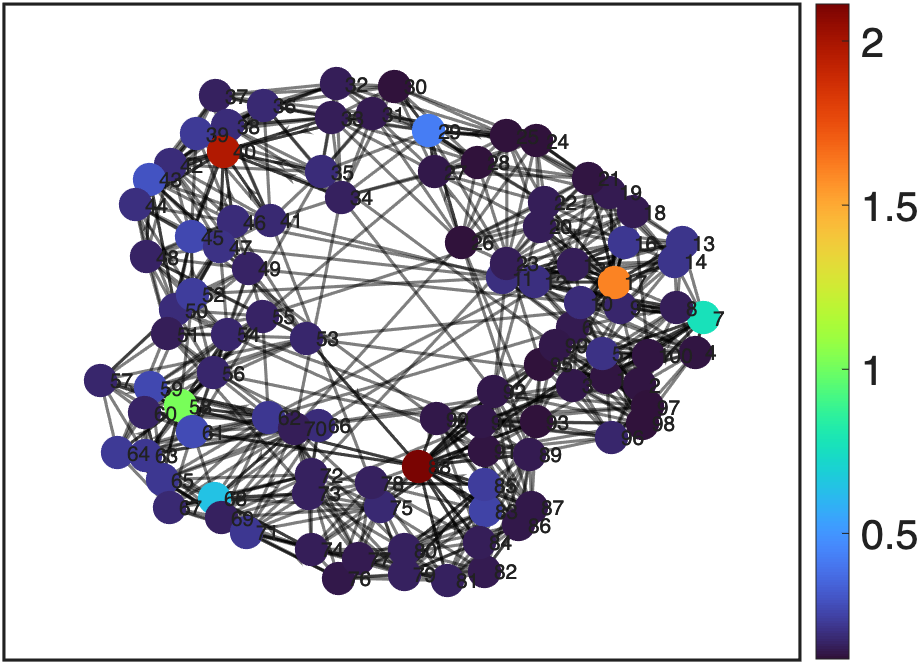}}
    
    \subfigure[]{\includegraphics[width=0.33\linewidth]{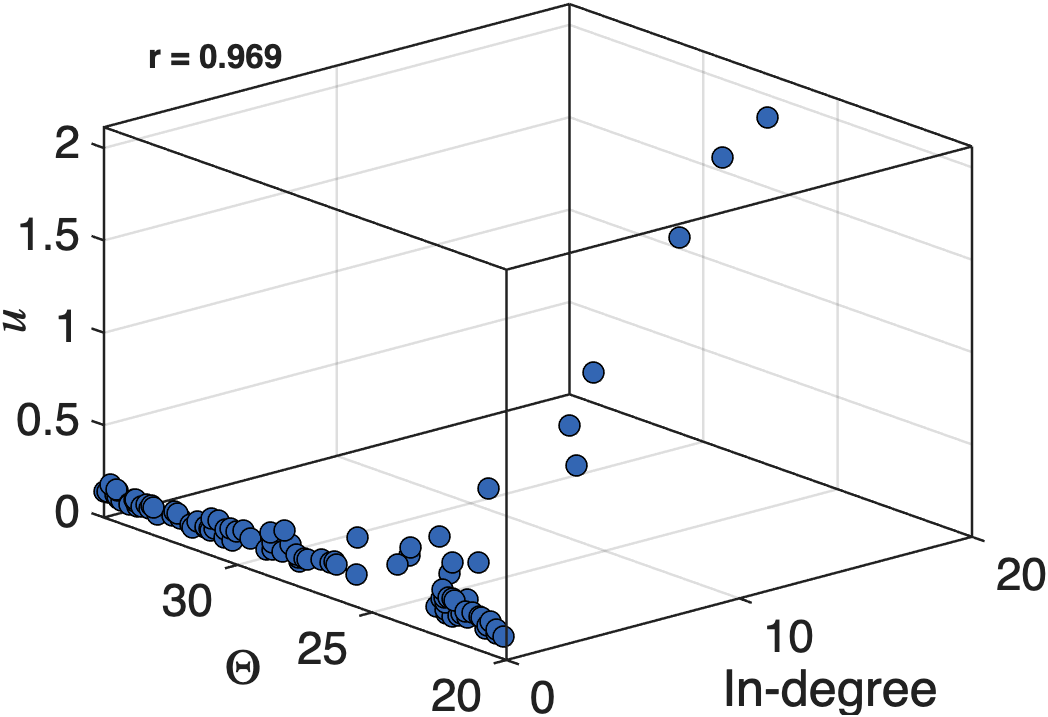}}
    \caption{(a) Random distribution of $\Theta$ across the network; (b) The accumulation of the density at the optimal nodes forming hotspots of population; (c) The dependence of the formation of hotspots based on the optimal thermal regime and ecological corridors. }
    \label{fig:network_random_distribution}
\end{figure}

\subsection{Directed movement of species towards optimal niche} \label{Sec:optimal_niche_prediction}
While modelling the distribution of environmental conditions over the graph, a reasonable assumption to consider is that a particular node feature is co-related to the feature of its neighboring nodes. Depending on the empirical construction of the network, this may not be the case also. Therefore, we study two cases in our model setup: First, where temperature is distributed randomly across the network, and second, where $T_i$ and $T_j$ are correlated if there is a link between these nodes, i.e, $\mathrm{A_{ij}}=1.$ For the later, it would mean that the temperature of a patch depends directly on the temperatures of neigbouring patches. Therefore, the habitat suitability of a species depends on its nearby nodes. With this assumption, we choose temperature to be the Gaussian Random field over the network \cite{rue2005gaussian}. A Gaussian random field on the network is defined as a multivariate Gaussian random vector
\[
\Theta \sim \mathcal{N}(\mu,\Sigma),
\]
where $\mu \in \mathbb{R}^N$ is the mean vector and $\Sigma \in \mathbb{R}^{N \times N}$ is a covariance matrix describing the spatial correlations between nodes with $\Sigma_{ij}$ measuring how strongly the environmental conditions at node $i$ are correlated with node $j.$ We assume a constant mean $\mu=\mu_0 \mathbf{1}$ and choose a covariance structure induced by the graph Laplacian $\mathrm{L}$ such that
\[
\Sigma = e^{-\tfrac{\sigma^2L}{2} },
\]
where $\sigma^2>0$ is a smoothing parameter controlling the spatial correlation length. Equivalently, the field $\Theta$ can be constructed by diffusing spatial white noise over the network,
\[
\Theta =e^{-\tfrac{\sigma^2L}{2} }\xi,
\qquad
\xi \sim \mathcal{N}(0,I),
\]
which represents the diffusion of $\Theta$ on the graph. The parameter $\sigma$ controls the degree of smoothing to the neighboring nodes where small values of $\sigma$ yield weakly correlated fields and larger values produce strong correlation among neighboring nodes. This construction ensures that nearby nodes in the network exhibit similar environmental conditions, while allowing heterogeneity at larger network distances. We then rescale the field to our interval interest $[a,b]$ of $\Theta$,
\[
\Theta_i \;\mapsto\; a + (b-a)\frac{\Theta_i-\min(\Theta)}{\max(\Theta)-\min(\Theta)}, \qquad i=1,\dots,N.
\]
We simulate our one-species model \eqref{eq:Compact_model} with temperature as the node feature and the per capita growth rate of the  depending on the ambient temperature as described in \eqref{eq:temperature_dependency}. The Fig.~\ref{fig:distribution_node_feature} (a) shows the gradual change in the temperature over the network, creating a heterogeneous environment with warmer regions on the right side and cooler regions on the left. This variation establishes a temperature gradient in the network which determined the habitat suitability of the species. We first simulate our model till $T=T_{max}$ without dispersal which gives us the local-node specific dynamics of the species driven by $F_i(u_i,\Theta_i)$. Therefore the nodes exhibit different intrinsic growth rates depending on how close their temperature is to the optimal value $\Theta^u_{\mathrm{opt}}$ (cf.~Fig.~\ref{fig:temperature_profile}). Nodes with temperatures close to $\Theta^u_{\mathrm{opt}}$ experience higher growth rate, whereas, nodes with temperatures far from the optimum show reduced growth rates and settle to steady state $u=0$ (cf.~Fig.~\ref{fig:distribution_node_feature}(b)). This creates cluster of nodes with with very high and very low population. We assume random dispersal along the possible corridors in the network, and since species can identify the environmental cue as temperature we assume an advection towards their optimal thermal nodes. This dispersal allows the species to persist beyond their optimal thermal nodes, however with low density and aggregating in few nodes which leads to a heterogeneous spatial pattern with population accumulation in fewer suitable patches (cf.~Fig.~\ref{fig:distribution_node_feature}(c)).

\begin{figure}[ht!]
    \centering
    \subfigure[Distribution of temperature]{\includegraphics[width=0.32\linewidth]{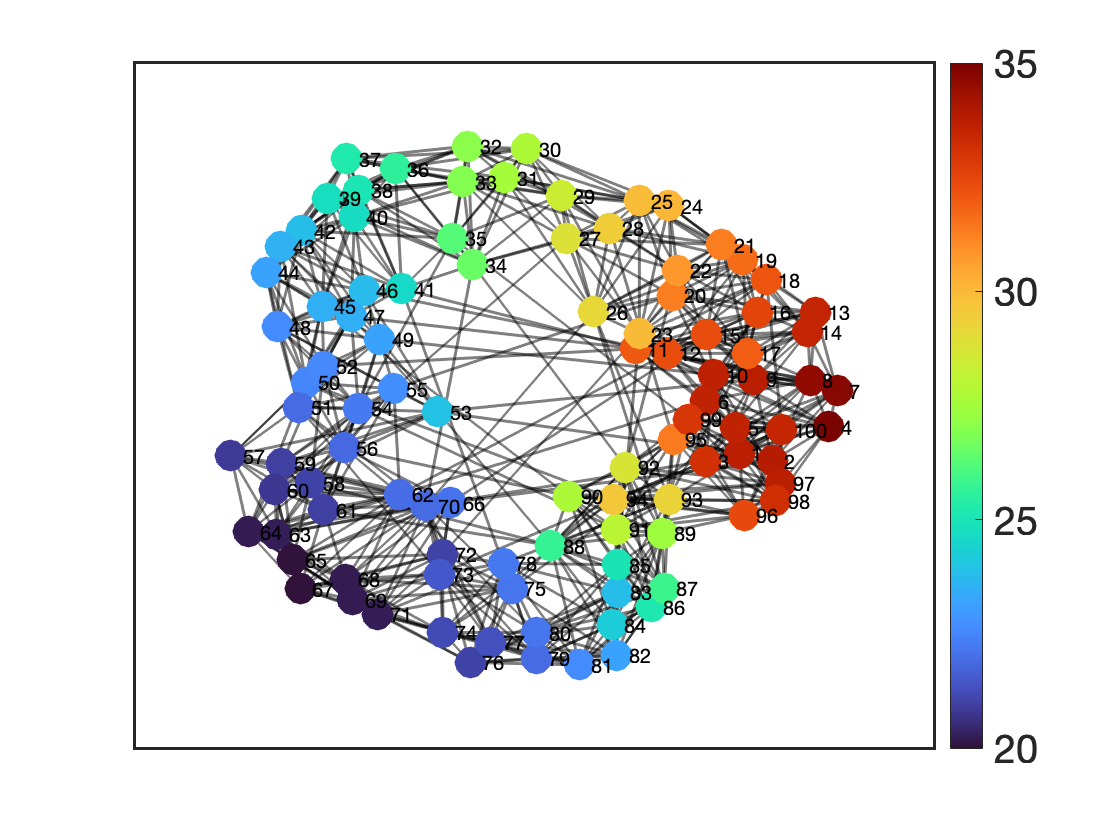}}  
\subfigure[Local dynamics of node]{\includegraphics[width=0.32\linewidth]{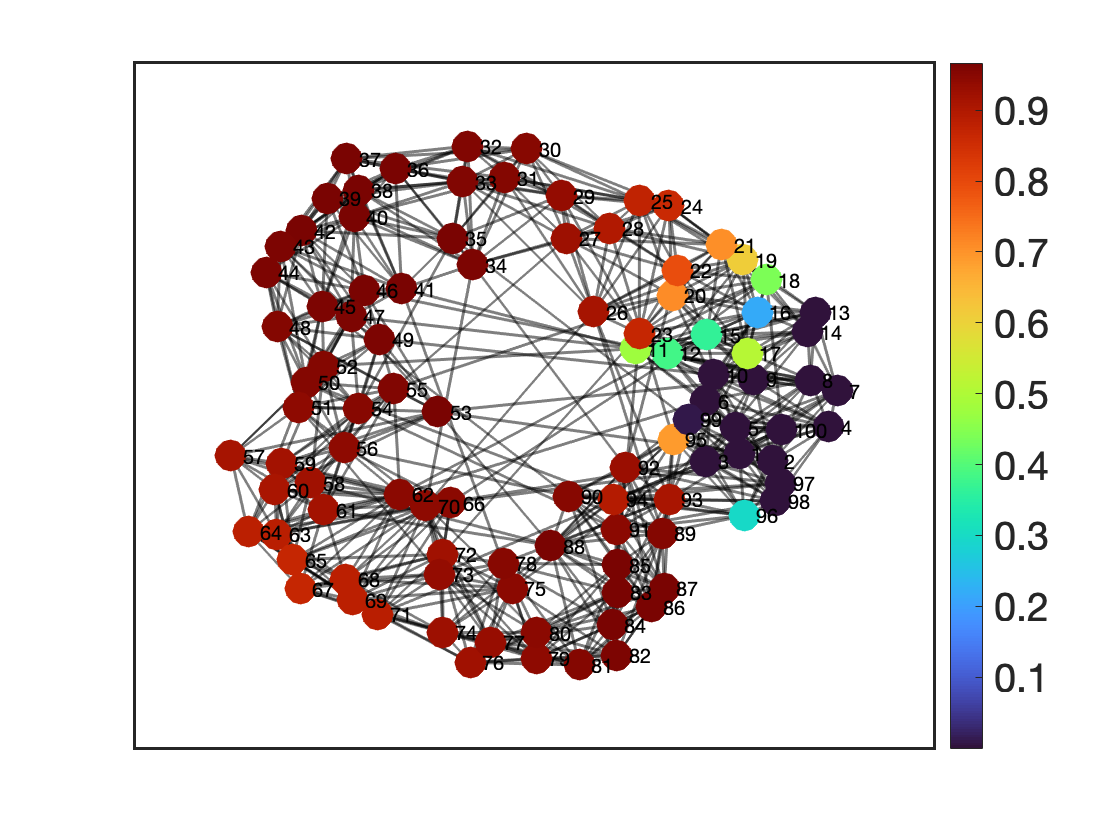}}
    %\subfigure[Diffusion]{\includegraphics[width=0.32\linewidth]%{Figures/one_species_reaction_diffusion_U_distribution.png}}
    \subfigure[Spatial distribution]{\includegraphics[width=0.32\linewidth]{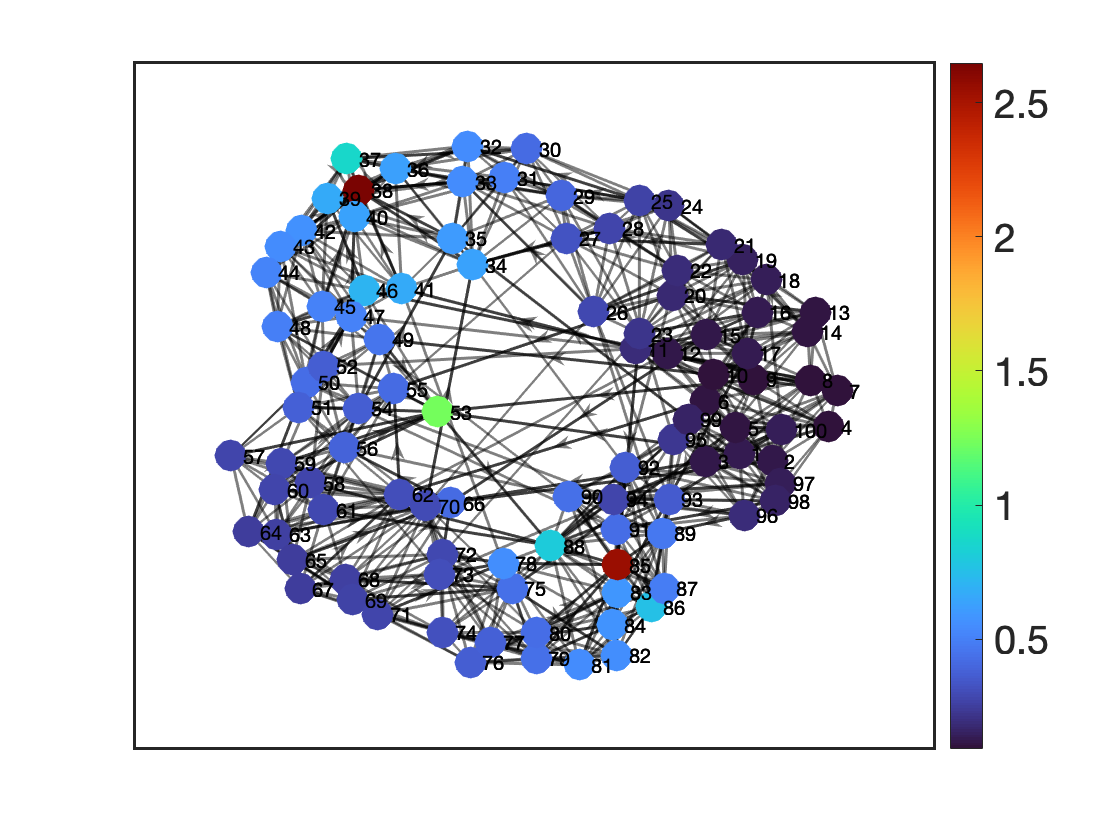}}
    % \subfigure[]{\includegraphics[width=0.32\linewidth]{Figures/one_species_reaction_diff_adv_indegree_u_.png}}

    % \subfigure[Reaction]{\includegraphics[width=0.3\linewidth]{Figures/one_species_reaction_U_vs_t_.png}}
    %  %\subfigure[reaction and diffusion]{\includegraphics[width=0.3\linewidth]{Figures/one_species_reaction_diff_U_vs_t.png}}  
    % \subfigure[reaction and diffusion and avection]{\includegraphics[width=0.3\linewidth]{Figures/one_species_reaction_diff_advec__U_vs_t.png}}

    \caption{(a) Distribution of the temperature over a synthetic Watts-Strogatz network with $N=100,\,K=7,\,\beta=0.1.$ (b) Simulation of the local dynamics without random dispersal and advection (c) The spatial distribution of the population from the warmer to cooler regions with accumulation in fewer optimal nodes due to advection. } %(d) The plot showing the abundance of $u$ with the indegree of each node, showing that optimal nodes with higher indegree accumulates the population and suboptimal nodes with lesser indegree cannot sustain population. 
    \label{fig:distribution_node_feature}
\end{figure}

\subsection{Interplay between advection strength and network topology}

\begin{figure}[ht!]
    \centering
      \includegraphics[width=17cm,height=9cm]{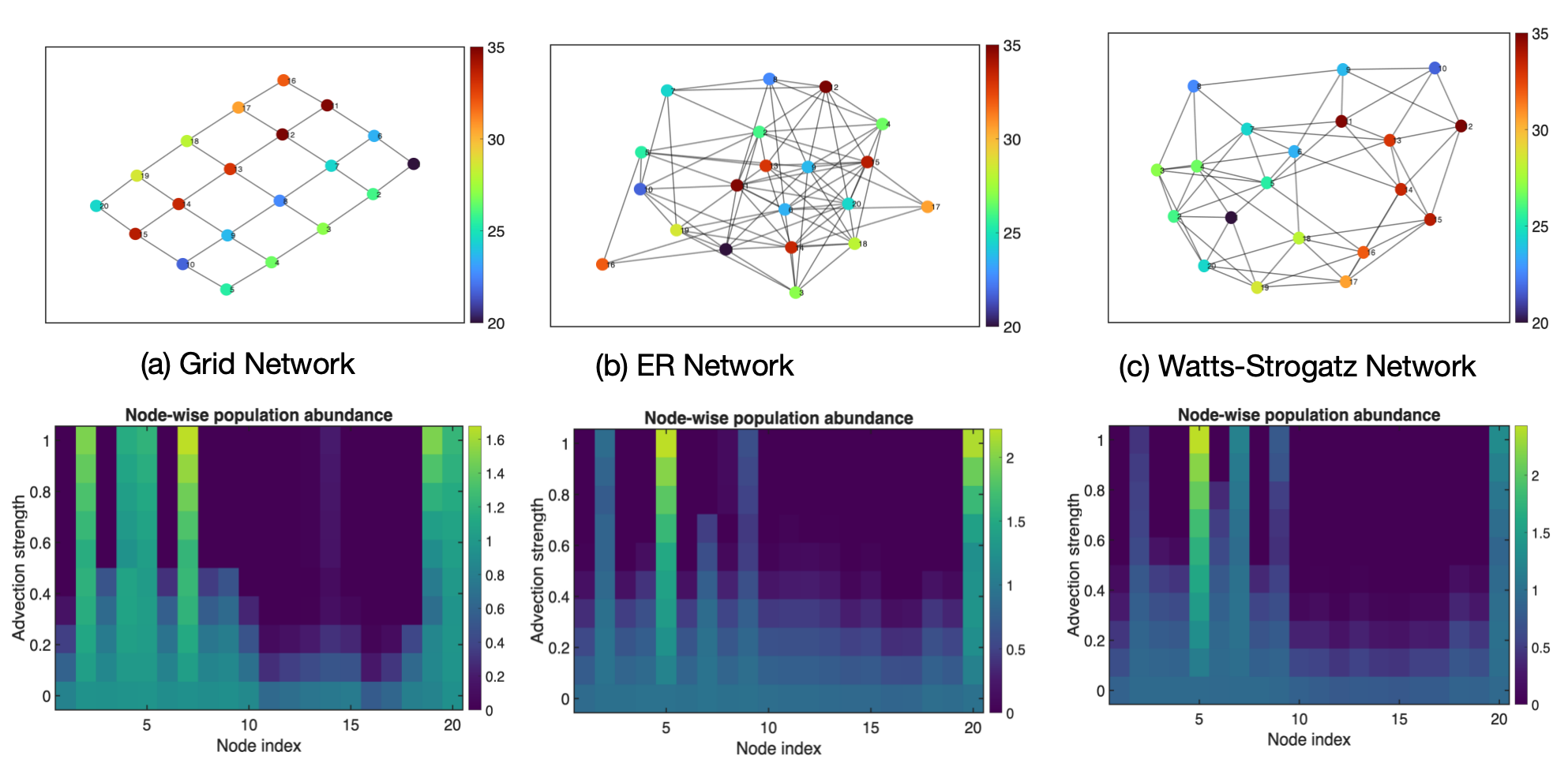}
       \caption{Upper panel: Different network topologies are described along with the temperature distribution in the network; Lower panel: The heatmap showing the population abundance at each node with increasing advection strength $\alpha.$ Strong advection create species hot spot in fewer nodes, whereas, weak advection increases persistence across the network. The temperature of each node is fixed across all three network topologies.}
    \label{fig:advection_strength}
\end{figure}

To investigate how the strength of the directed movement interacts with the spatial structure of the network, we examined the population dynamics under different advection strengths in three distinct network topologies. In the system \eqref{eq:Compact_model} when $\alpha$ is small, species dispersal is dominated by random diffusion. This regime corresponds to the weak advection. Whenever the advection $\alpha$ dominate the random diffusion, we say that the movement is strongly driven by the temperature gradient, hence strong advection. We simulated our model for three different network topologies with N=20 nodes: (a) grid network, (b) Erd\H{o}s--R\'enyi network with $p=0.4$ \cite{erdos1961evolution}, (c) Watts-Strogatz Network with $k=3,\,\beta=0.4.$ Our analysis shows that across all the network topologies considered, increasing the advection strength causes range contraction of population. When the advection strength is weak, the population persists across the entire network, resulting in a relatively homogeneous spatial distribution. This gives a distributed persistence even in the suboptimal and non-optimal nodes. Whereas, for strong advection, individuals were increasingly transported toward nodes with environmental conditions closer to the species' thermal optimum. This leads to aggregation of population in fewer nodes, while most nodes experience local extinction (we set the tolerance $10^{-6})$. Therefore, with increasing strength of advection, the spatial distribution of species shifts from distributed persistence to an aggregated persistence regime. This qualitative mechanism holds true across different network topologies as shown in Fig.~\ref{fig:advection_strength} for a fixed diffusivity coefficient $d=0.3.$ For advection strength $\alpha \in [0,0.4]$ the species distributes even in the non-optimal regions, whereas, higher advection of $\alpha=1$ show aggregated population in fewer nodes of optimal temperature with sufficient network connectivity.

\subsection{Impact of corridor loss on directed movement}

Modelling the directed movement over the network facilitates the movement of species population towards the optimal ecological patch. When the network is intact, these patches support the maximum population. However, due to the increase in anthropogenic activities, there has been a significant loss of ecological corridors over the last few decades \cite{haddad2015habitat}. As a result, the species might be able to efficiently track their optimal habitat patch. Therefore, it is interesting to note how the population is redistributed, if at all, in a network if the link to an optimal or fewer optimal nodes has been destroyed.
\begin{figure}[ht!]
    \centering
    \subfigure[]{\includegraphics[width=8cm,height=5.15cm]{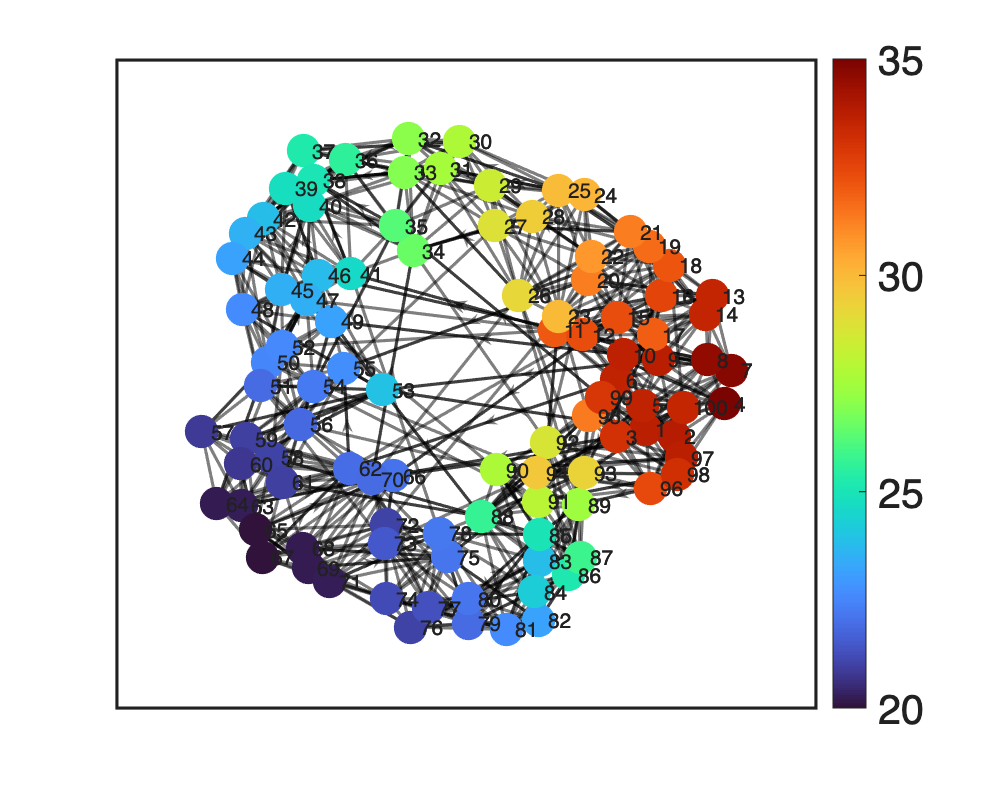}}
    \subfigure[]{\includegraphics[width=7.8cm,height=5.15cm]{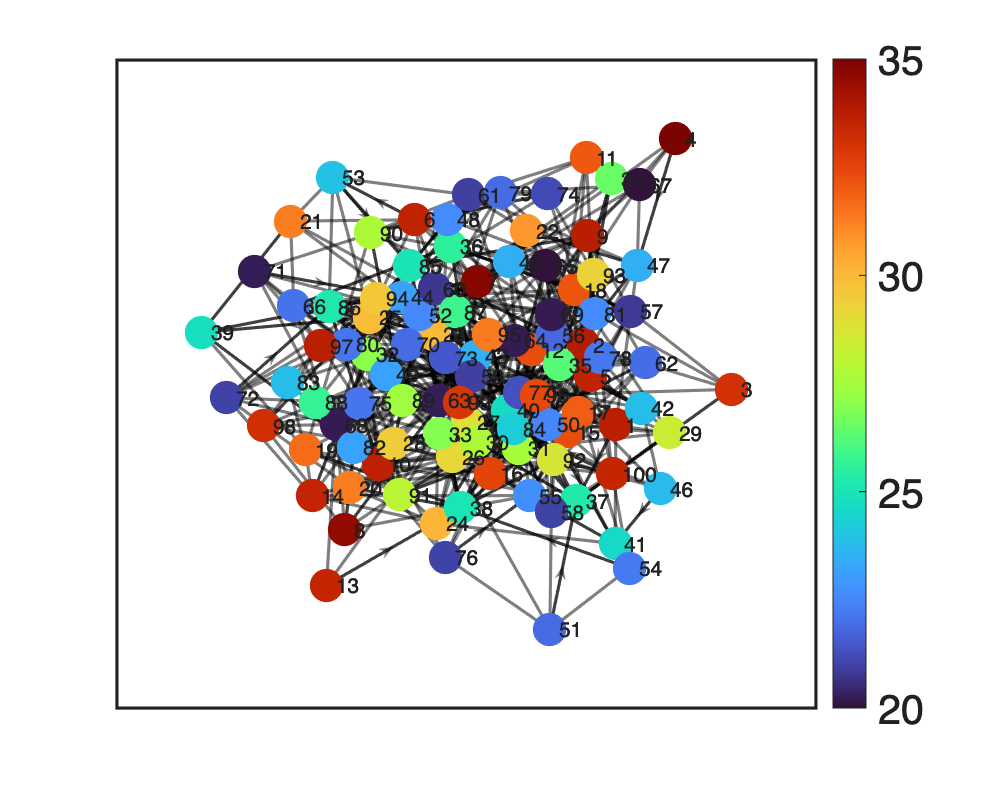}}
    \subfigure[]{
    \includegraphics[width=8cm,height=5.15cm]{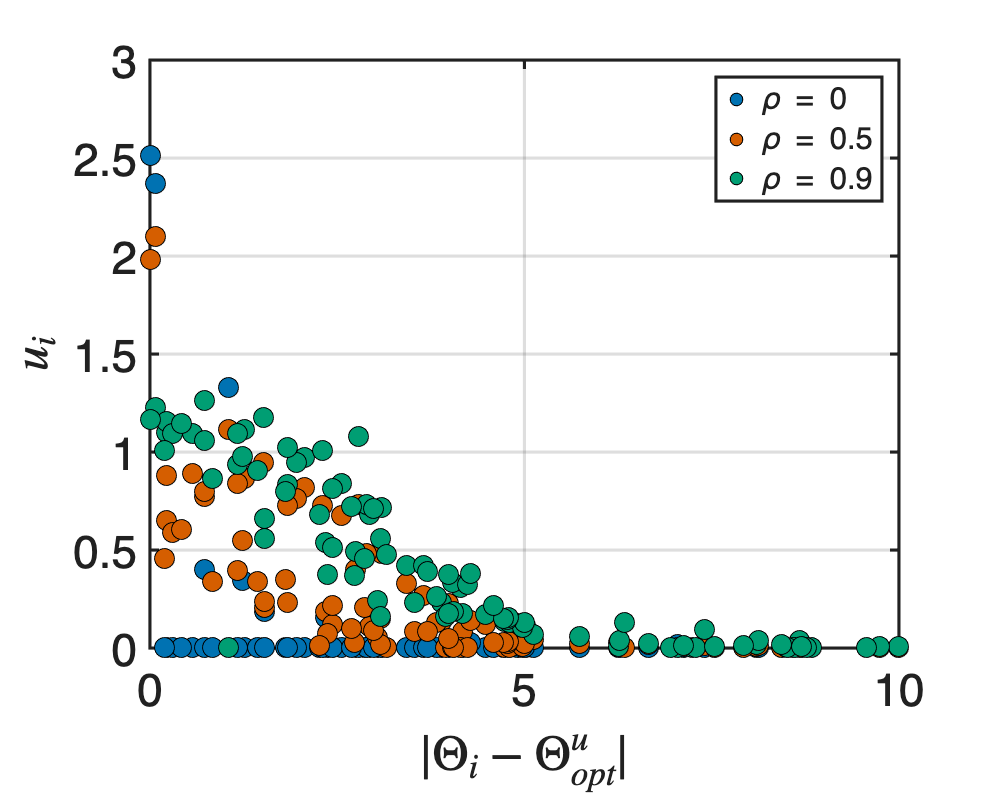}}
    \subfigure[]{
    \includegraphics[width=8cm,height=5.15cm]{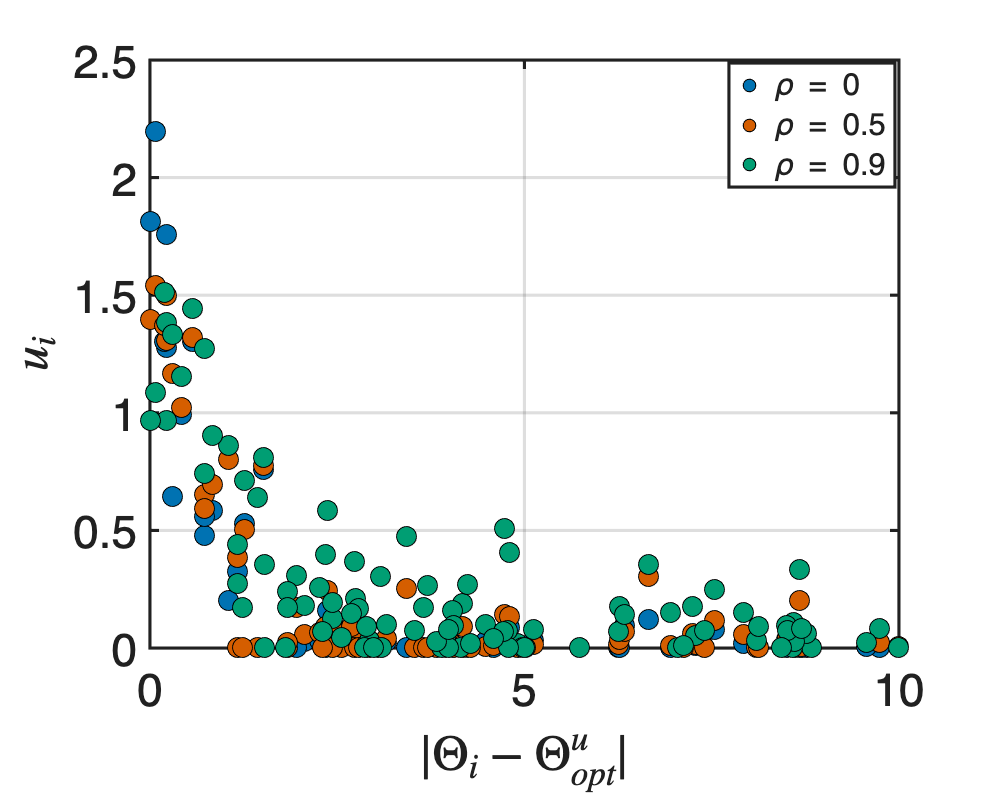}}
    \caption{Distribution of temperature feature on the two network topologies: (a) Watts-Strogatz network with $N=100,\,K=7,\,\beta=0.1,$ (b)  Erd\H{o}s--R\'enyi network with $p=0.1,$ (c,d) Distribution of population in the corresponding network is plotted against the distance from their thermal optima $|\Theta_i-\mathrm{\Theta^u_{opt}}|$ for $\rho=0,\,0.5,\,0.9.$ The distribution indicates that higher $\rho$ forces the population to occupy the neighboring suboptimal habitats.}
    \label{fig:Corridor_loss_experiment}
\end{figure}

To understand the effect, we simulated our model with the same network as discussed in Section \ref{Sec:optimal_niche_prediction} of $N=100$ nodes, and with an ER-network \cite{erdos1961evolution}. The distribution of the temperature is considered as in Fig.~\ref{fig:distribution_node_feature}(a). We define the corridor loss fraction $(\rho)$ as
\begin{equation}
    \rho = \frac{\text{Number of edges removed}}{\text{Total number of edges}},
\end{equation}
implying that $\rho\approx0$ gives the connected network, and $\rho\approx 1$ gives the completely fragmented network. Since for a high advection, we observed that the species accumulate in the best habitable nodes, or the nodes which are closest to their thermal optimum. In our simulation, we sequentially removed $\rho$ edges connecting to the optimal nodes. This will allow us to study the impact of loss of corridors connected to the optimal habitable nodes. For an equal diffusivity and advection strength, the net dispersal for $\rho=0$ is determined by $\mathrm{M}=d (\mathrm{L}
+\mathrm{A_u^{adv}(\Theta)}).$ When the directed edges ($i\rightarrow j$) towards the optimal node ($j$) are removed sequentially, the entries of the matrix $\mathrm{L_{ij}}=\mathrm{L_{ji}}=0$ and $\mathrm{\Tilde{a}_{ij}}=0.$ Therefore, the growth in the optimal nodes persists majorly from the local dynamics $F(u_i,\Theta_i).$ Therefore, even if the node becomes isolated for $\rho \approx 1,$ the population still persists locally, however, at a lower abundance. Further, the neighboring nodes of $j\,(N(j)),$ cannot disperse and stays there even if the environmental conditions are unfavorable. This implies that the suboptimal neighboring nodes $N(j)$ of the optimal node $j$ have higher population density. Depending on the number of optimal nodes the network has, and the number of edges towards the optimal nodes are removed the population tends to accumulate in the neighboring nodes. This shows the spread of population in the nodes much farther from their optimal regime. Fig.~\ref{fig:Corridor_loss_experiment} shows the result for two different network and different fractions of corridor loss $(\rho).$ We observe in both the network topology that species abundance is concentrated in the optimal nodes for $\rho=0,$ where $|\Theta_i-\mathrm{\Theta^u_{opt}}|\approx 0.$ In contrast, for $\rho=0.5,\,0.9$ the peak abundance though observed in the optimal nodes, but much lesser than for $\rho=0$ and the population is seen to spread till $|\Theta_i-\mathrm{\Theta^u_{opt}}|\approx 5.$ This abundance is dominantly from the local dynamics of the neighboring nodes. These results indicate that habitat fragmentation does not cause immediate extinction, rather, forces populations to occupy a broader range of environmental conditions.

\section{Discussion}

In literature, the directional movement of species has been studied mostly in the form of weighted diffusion \cite{prima2018combining}. However, such approaches fundamentally represent biased random dispersal and do not explicitly capture flow-driven movement. In this work, we developed a graph-based reaction-diffusion-advection framework that explicitly incorporates the directional movement of species across heterogeneous landscapes. The heterogeneity in the landscape is mediated through the environmental gradient between the nodes in the network. By embedding environmentally driven flows into a network structure, the model moves beyond traditional diffusion-based approaches and captures a key, yet underexplored, mechanism of biased dispersal along environmental gradients.

The empirical evidences of species responding to environmental cues \cite{chen2011rapid,lenoir2020species}, and particularly to climate warming motivated us to develop this formulation. By representing the patchy landscape as a network, the directional flow of the species are modelled based on the temperature gradient across the network. A key finding of this work is that, unlike purely diffusive movement which homogenizes population across the network, advection induces asymmetry forcing soure-sink like dynamics \cite{bisschop2019transient} and thus alters the spatial distribution patterns. The persistence of species in a particular patch is not solely determined by local patch quality, but by the ecological corridors connecting the neighboring patches. Our advection formulation efficiently captures the dispersal of species from warmer to cooler regions, thus creating clusters of nodes with high and low population densities (cf.~Fig.~\ref{fig:distribution_node_feature}). This confirms that species always migrate to patches where the temperature conditions are closest to their thermal niche. Thus creating a heterogeneous spatial pattern on the network with more accumulation of species in fewer optimal patches. This has direct implications for identifying critical patches for biodiversity conservation especially in fragmented landscapes.

Our analysis highlights the interaction between network topology and advective strength. Strong advection can arise during the extreme events like heatwaves, high precipitation, floods, etc, which forces a rapid redistribution of the population across the landscape. By simulating our model for various network topologies we confirm that indeed the population is driven towards fewer optimal nodes, thereby, increasing the risk of local extinction (cf.~Fig.~\ref{fig:advection_strength}). In such contexts, advection becomes the dominant mechanism shaping species distribution and persistence. In a fully connected network, species can efficiently migrate toward their thermally optimal patch through ecological corridors. However, with an increased destruction in ecological corridors, the ability of species to track their optimal thermal niches becomes increasingly constrained. Our simulation support this hypothesis and show that with increasing loss of corridors, the species do not exhibit immediate extinction, instead, they are forced to occupy patches with temperature farther from their optimal regime. Notably, the population survive in these non-optimal patches with reduced density. This persistence may indicate the adaptability of species to a new thermal regime \cite{bennett2019integrating}.    

From a mechanistic perspective, the explicit formulation of advection on graphs provides a bridge between continuous PDE based diffusion-advection models and discrete network representations of landscapes. To the best of our knowledge, an explicit formulation of advection within ecological networks remains largely unexplored.  It would be useful for  studying a wide range of ecological processes, including range shifts, invasion dynamics, and climate-driven redistribution, particularly, when integrated with high-resolution environmental and movement data. Moreover, it opens up several questions for further research. For instance, integrating the proposed advection framework with the existing species distribution models can provide a more mechanistic basis for predicting species distributions under changing climatic conditions.

\section*{Acknowledgments}
P.R.C acknowledges the Axis Bank Centre for Mathematics and Computing for the financial assistance to carry out this research work in the Department of Computational and Data Sciences, Indian Institute of Science, Bengaluru.

\section*{Conflict of interest}
The authors declare that there are no conflicts of interest.

\section*{Data availability}
The code supporting the findings of this study will be made publicly available upon publication.

\bibliographystyle{unsrt}
\bibliography{References}

\end{document}